\documentclass[a4paper,11pt,fleqn]{article}
\usepackage{pstricks}
\usepackage{amsmath}
\usepackage{amssymb}
\usepackage{pifont}
\usepackage{theorem} 
\usepackage{euscript}
\usepackage{exscale,relsize}
\usepackage{epic,eepic}
\usepackage{multicol}
\usepackage{makeidx}
\usepackage{srcltx}
\usepackage{xcolor}
\usepackage{url}
\usepackage{charter}
\usepackage{booktabs,subfig,graphicx}
\usepackage[normalem]{ulem}     
\renewcommand{\leq}{\ensuremath{\leqslant}}

\newcommand{\minimize}[2]{\ensuremath{\underset{\substack{{#1}}}%
{\text{minimize}}\;\;#2 }}

\newcommand{\scal}[2]{{\left\langle{{#1}\mid{#2}}\right\rangle}}

\newcommand{\menge}[2]{\big\{{#1}~\big |~{#2}\big\}}

\newcommand{\KKK}{\ensuremath{\boldsymbol{\mathcal K}}}

\newcommand{\HH}{\ensuremath{{\mathcal H}}}
\newcommand{\GG}{\ensuremath{{\mathcal G}}}

\newcommand{\Sum}{\ensuremath{\displaystyle\sum}}
\newcommand{\emp}{\ensuremath{{\varnothing}}}

\newcommand{\Id}{\ensuremath{\text{\rm Id}}\,}
\newcommand{\cart}{\ensuremath{\raisebox{-0.5mm}{\mbox{\LARGE{$\times$}}}}}

\newcommand{\BL}{\ensuremath{\EuScript B}\,}
\newcommand{\RPP}{\ensuremath{\left]0,+\infty\right[}}

\newcommand{\RX}{\ensuremath{\left]-\infty,+\infty\right]}}

\newcommand{\NN}{\ensuremath{\mathbb N}}

\newcommand{\weakly}{\ensuremath{\:\rightharpoonup\:}}

\newcommand{\ran}{\ensuremath{\text{\rm ran}\,}}

\newcommand{\zer}{\ensuremath{\text{\rm zer}\,}}
\newcommand{\pinf}{\ensuremath{{+\infty}}}

\newcommand{\dom}{\ensuremath{\text{\rm dom}\,}}
\newcommand{\prox}{\ensuremath{\text{\rm prox}}}

\newcommand{\gra}{\ensuremath{\text{\rm gra}\,}}

\newtheorem{theorem}{Theorem}[section]
\newtheorem{lemma}[theorem]{Lemma}
\newtheorem{corollary}[theorem]{Corollary}

\theoremstyle{plain}{\theorembodyfont{\rmfamily}%
}
\theoremstyle{plain}{\theorembodyfont{\rmfamily}%
}
\theoremstyle{plain}{\theorembodyfont{\rmfamily}%
\newtheorem{remark}[theorem]{Remark}}
\theoremstyle{plain}{\theorembodyfont{\rmfamily}%
}
\theoremstyle{plain}{\theorembodyfont{\rmfamily}%
}
\theoremstyle{plain}{\theorembodyfont{\rmfamily}%
}
\theoremstyle{plain}{\theorembodyfont{\rmfamily}%
\newtheorem{problem}[theorem]{Problem}}

\numberwithin{equation}{section}
\begin{document}

\title{\sffamily\huge A Primal-Dual Method of Partial Inverses\\
for Composite Inclusions\footnote{Contact author: 
P. L. Combettes, \ttfamily{plc@math.jussieu.fr},
phone: +33 1 4427 6319, fax: +33 1 4427 7200.}}

\author{
Maryam A. Alghamdi,$\!^{\sharp}$
~Abdullah Alotaibi,$\!^{\natural}$
~Patrick L. Combettes,$\!^{\flat,\natural}$~ and 
~Naseer Shahzad$\,^\natural$
\\[5mm]
\small $^\sharp$King Abdulaziz University\\
\small Department of Mathematics\\
\small Sciences Faculty for Girls, P. O. Box 4087\\
\small Jeddah 21491, Saudi Arabia\\[5mm]
\small $^\natural$King Abdulaziz University\\
\small Department of Mathematics, P. O. Box 80203\\
\small Jeddah 21859, Saudi Arabia\\[5mm]
\small $^\flat$UPMC Universit\'e Paris 06\\
\small Laboratoire Jacques-Louis Lions -- UMR CNRS 7598\\
\small 75005 Paris, France\\[4mm]
}
\date{~}

\maketitle
\thispagestyle{empty}

\vskip 8mm

\begin{abstract} 
\noindent
Spingarn's method of partial inverses has found many applications
in nonlinear analysis and in optimization. We show that it can be
employed to solve composite monotone inclusions in duality, 
thus opening a new range of applications for the partial 
inverse formalism. The versatility of the resulting 
primal-dual splitting algorithm is illustrated through applications 
to structured monotone inclusions and optimization.
\end{abstract} 

\vskip 10mm


\vskip 6mm


\maketitle

\newpage

\section{Introduction}

Let $\HH$ be a real Hilbert space, let $A\colon\HH\to 2^{\HH}$ be a 
set-valued operator, let 
$\gra A=\menge{(x,u)\in\HH\times\HH}{u\in Ax}$ denote the graph of
$A$, let $V$ be a closed vector subspace of $\HH$, and let 
$P_V$ and $P_{V^\bot}$ denote respectively the projectors 
onto $V$ and onto its orthogonal complement. The partial 
inverse of $A$ with respect to $V$ is defined through 
\begin{equation}
\label{e:2013-07-25a}
\gra A_V=\menge{(P_Vx+P_{V^\bot}u,P_Vu+P_{V^\bot}x)}
{(x,u)\in\gra A}.
\end{equation}
This operator, which was introduced by Spingarn in \cite{Spin83}, 
can be regarded as an intermediate object between $A$ and $A^{-1}$.
A key result of \cite{Spin83} is that, if $A$ is maximally 
monotone, problems of the form
\begin{equation}
\label{e:999fgr623l17-27h}
\text{find}\;\;x\in V\;\;\text{and}\;\;
u\in V^\bot\;\;\text{such that}\;\;u\in Ax
\end{equation}
can be solved by applying the proximal point algorithm algorithm
\cite{Rock76} to the partial inverse $A_V$. The resulting
algorithm, known as the \emph{method of partial inverses}, has 
been applied to various problems in nonlinear analysis and 
optimization; see, e.g., \cite{Jat104,Bura06,Joca09,Ecks92,%
Idri89,Lema89,Mahe95,Penn02,Spin83,Spin85,Spin87-2}. 
The goal of the present paper is to propose a new range of 
applications of the method of partial inverses by showing that 
it can be applied to solving the following type of monotone 
inclusions in duality.

\begin{problem}
\label{prob:1}
Let $\HH$ and $\GG$ be real Hilbert spaces, let 
$A\colon\HH\to 2^{\HH}$ and $B\colon\GG\to 2^{\GG}$ be maximally 
monotone operators, and let $L\colon\HH\to\GG$ be a bounded linear 
operator. Solve the primal inclusion
\begin{equation}
\label{e:primal}
\text{find}\;\;x\in\HH\;\;\text{such that}\;\; 
0\in Ax+L^*BLx
\end{equation}
together with the dual inclusion
\begin{equation}
\label{e:dual}
\text{find}\;\;v\in\GG\;\;\text{such that}\;\; 
0\in -LA^{-1}(-L^*v)+B^{-1}v.
\end{equation}
\end{problem}

The operator duality described in \eqref{e:primal}--\eqref{e:dual}
is an extension of the classical Fenchel-Rockafellar duality setting
for functions \cite{Rock67} which has been studied in particular in 
\cite{Baus12,Siop11,Ecks99,Penn00,Robi99}.
Our main result shows that, through a suitable reformulation, 
Problem~\ref{prob:1} can be reduced to a problem of the form 
\eqref{e:999fgr623l17-27h} and that the method of partial 
inverses applied to the latter leads to a splitting algorithm in 
which the operators $A$, $B$, and $L$ are used separately. 

The remainder of the paper is organized as follows.
In Section~\ref{sec:2}, we revisit the method of partial inverses 
and propose new convergence results. Our method of partial inverses
for solving Problem~\ref{prob:1} is presented in 
Section~\ref{sec:3}, together with applications. An alternative 
implementation of this method is proposed in Section~\ref{sec:4},
where further applications are provided. Section~\ref{sec:5}
contains concluding remarks.

{\bfseries Notation.}  
The scalar product of a real Hilbert space $\HH$ is denoted by 
$\scal{\cdot}{\cdot}$ and the associated norm by $\|\cdot\|$;
$\weakly$ and $\to$ denote, respectively, weak and 
strong convergence, and $\Id$ is the identity operator. 
We denote by $\BL(\HH,\GG)$ the class of bounded linear operators
from $\HH$ to a real Hilbert space $\GG$.
Let $A\colon\HH\to 2^{\HH}$.
The inverse $A^{-1}$ of $A$ is defined 
via $\gra A^{-1}=\menge{(u,x)\in\HH\times\HH}{(x,u)\in\gra A}$,
$J_A=(\Id+A)^{-1}$ is the resolvent of $A$, 
$\ran A=\bigcup_{x\in\HH}Ax$ is the range of $A$, and
$\zer A=\menge{x\in\HH}{0\in Ax}$ is the set of zeros of $A$.
We denote by $\Gamma_0(\HH)$ the class of lower semicontinuous 
convex functions $f\colon\HH\to\RX$ such that
$\dom f=\menge{x\in\HH}{f(x)<\pinf}\neq\emp$. Let
$f\in\Gamma_0(\HH)$. The conjugate of $f$ is the function
$f^*\in\Gamma_0(\HH)$ defined by $f^*\colon u\mapsto
\sup_{x\in\HH}(\scal{x}{u}-f(x))$.
For every $x\in\HH$, $f+\|x-\cdot\|^2/2$ possesses a 
unique minimizer, which is denoted by $\prox_fx$. We have 
\begin{equation}
\label{e:prox2}
\prox_f=J_{\partial f},\quad\text{where}\quad
\partial f\colon\HH\to 2^{\HH}\colon x\mapsto
\menge{u\in\HH}{(\forall y\in\HH)\;\:\scal{y-x}{u}+f(x)\leq f(y)} 
\end{equation}
is the subdifferential of $f$. 
See \cite{Livre1} for background on monotone operators and 
convex analysis.

\section{A method of partial inverses}
\label{sec:2}

Throughout this section, $\KKK$ denotes a real Hilbert space.
We establish the convergence of a relaxed, 
error-tolerant method of partial inverses. 
First, we review some results about partial inverses and
the proximal point algorithm.

\begin{lemma}{\rm\cite[Section~2]{Spin83}}
\label{l:2013-07-26a}
Let $\boldsymbol{A}\colon\KKK\to 2^{\KKK}$ be maximally monotone,
let $\boldsymbol{V}$ be a closed vector subspace of 
$\KKK$, and let $\boldsymbol{z}\in\KKK$. Then 
$\boldsymbol{A}_{\boldsymbol{V}}$ is maximally monotone and 
$\boldsymbol{z}\in\zer\boldsymbol{A}_{\boldsymbol{V}}$ 
$\Leftrightarrow$ $(P_{\boldsymbol{V}}\boldsymbol{z},
P_{\boldsymbol{V}^\bot}\boldsymbol{z})\in\gra\boldsymbol{A}$.
\end{lemma}

\begin{lemma}
\label{l:999fgr623l17-29}
Let $\boldsymbol{A}\colon\KKK\to 2^{\KKK}$ be maximally monotone,
let $\boldsymbol{V}$ be a closed vector subspace of 
$\KKK$, let $\boldsymbol{z}\in\KKK$, and let 
$\boldsymbol{p}\in\KKK$. Then 
$\boldsymbol{p}=J_{\boldsymbol{A}_{\boldsymbol{V}}}\boldsymbol{z}$ 
$\Leftrightarrow$ $P_{\boldsymbol{V}}\boldsymbol{p}+
P_{\boldsymbol{V}^\bot}(\boldsymbol{z}-\boldsymbol{p})
=J_{\boldsymbol{A}}\boldsymbol{z}$.
\end{lemma}
\begin{proof}
This result, which appears implicitly in \cite[Section~4]{Spin83}, 
follows from the equivalences
\begin{eqnarray}
\label{e:h5-hGEWD12a}
P_{\boldsymbol{V}}\boldsymbol{p}+P_{\boldsymbol{V}^\bot}
(\boldsymbol{z}-\boldsymbol{p})=J_{\boldsymbol{A}}\boldsymbol{z}
&\Leftrightarrow&
\big(P_{\boldsymbol{V}}\boldsymbol{p}+P_{\boldsymbol{V}^\bot}
(\boldsymbol{z}-\boldsymbol{p}),
\boldsymbol{z}-P_{\boldsymbol{V}}\boldsymbol{p}-
P_{\boldsymbol{V}^\bot}(\boldsymbol{z}-\boldsymbol{p})\big)\in\gra
\boldsymbol{A}\nonumber\\
&\Leftrightarrow&
\big(P_{\boldsymbol{V}}\boldsymbol{p}+P_{\boldsymbol{V}^\bot}
(\boldsymbol{z}-\boldsymbol{p}),
P_{\boldsymbol{V}}(\boldsymbol{z}-\boldsymbol{p})+
P_{\boldsymbol{V}^\bot}\boldsymbol{p}\big)\in\gra\boldsymbol{A}
\nonumber\\
&\Leftrightarrow&
(\boldsymbol{p},\boldsymbol{z}-\boldsymbol{p})\in\gra
\boldsymbol{A}_{\boldsymbol{V}}
\nonumber\\
&\Leftrightarrow&
\boldsymbol{p}=J_{\boldsymbol{A}_{\boldsymbol{V}}}\boldsymbol{z}, 
\end{eqnarray}
where we have used \eqref{e:2013-07-25a}.
\end{proof}

\begin{lemma}{\rm\cite[Remark~2.2(vi)]{Joca09}}
\label{l:2009}
Let $\boldsymbol{B}\colon\KKK\to 2^{\KKK}$ be a maximally monotone
operator such that $\zer\boldsymbol{B}\neq\emp$, let 
$\boldsymbol{z}_0\in\KKK$, let $(\boldsymbol{c}_n)_{n\in\NN}$ be a 
sequence in $\KKK$, and let $(\lambda_n)_{n\in\NN}$ be a sequence 
in $\left]0,2\right[$. Suppose that 
$\sum_{n\in\NN}\lambda_n(2-\lambda_n)=\pinf$ and 
$\sum_{n\in\NN}\lambda_n\|\boldsymbol{c}_n\|<\pinf$, and set 
\begin{equation}
\label{e:alg04}
(\forall n\in\NN)\quad
\boldsymbol{z}_{n+1}=\boldsymbol{z}_n+\lambda_n
\big(J_{\boldsymbol{B}}\boldsymbol{z}_n+
\boldsymbol{c}_n-\boldsymbol{z}_n\big).
\end{equation}
Then $J_{\boldsymbol{B}}\boldsymbol{z}_n-\boldsymbol{z}_n\to
\boldsymbol{0}$ and $(\boldsymbol{z}_n)_{n\in\NN}$ converges 
weakly to a point in $\zer\boldsymbol{B}$.
\end{lemma}

The next theorem analyzes a method of partial inverses and 
extends the results of \cite{Ecks92} and \cite{Spin83}.

\begin{theorem}
\label{t:2013-08-19}
Let $\boldsymbol{A}\colon\KKK\to 2^{\KKK}$ be a maximally monotone
operator, let $\boldsymbol{V}$ be a closed vector subspace of 
$\KKK$, let $(\lambda_n)_{n\in\NN}$ be a sequence in 
$\left]0,2\right[$, and let $(\boldsymbol{e}_n)_{n\in\NN}$ be a 
sequence in $\KKK$ such that 
$\sum_{n\in\NN}\lambda_n(2-\lambda_n)=\pinf$ and 
$\sum_{n\in\NN}\lambda_n\|\boldsymbol{e}_n\|<\pinf$.
Suppose that the problem 
\begin{equation}
\label{e:999fgr623l17-27a}
\text{find}\;\;\boldsymbol{x}\in
\boldsymbol{V}\;\;\text{and}\;\;\boldsymbol{u}
\in\boldsymbol{V}^\bot\;\;\text{such that}\;\;\boldsymbol{u}\in
\boldsymbol{A}\boldsymbol{x}
\end{equation}
has at least one solution, let $\boldsymbol{x}_0\in\boldsymbol{V}$, 
let $\boldsymbol{u}_0\in\boldsymbol{V}^\bot$, and set
\begin{equation}
\label{e:999fgr623l18-20b}
(\forall n\in\NN)\quad 
\begin{array}{l}
\left\lfloor
\begin{array}{l}
\boldsymbol{p}_n=J_{\boldsymbol{A}}(\boldsymbol{x}_n+
\boldsymbol{u}_n)+\boldsymbol{e}_n\\
\boldsymbol{r}_n=\boldsymbol{x}_n+\boldsymbol{u}_n-
\boldsymbol{p}_n\\
\boldsymbol{x}_{n+1}=
\boldsymbol{x}_n-\lambda_n P_{\boldsymbol{V}}\boldsymbol{r}_n\\
\boldsymbol{u}_{n+1}=
\boldsymbol{u}_n-\lambda_n P_{\boldsymbol{V}^\bot}
\boldsymbol{p}_n.
\end{array}
\right.\\[2mm]
\end{array}
\end{equation}
Then the following hold:
\begin{enumerate}
\item
\label{t:2013-08-19i}
$P_{\boldsymbol{V}}(\boldsymbol{p}_n-\boldsymbol{e}_n)-
\boldsymbol{x}_n\to\boldsymbol{0}$ ~and~ 
$P_{\boldsymbol{V}^\bot}(\boldsymbol{r}_n+\boldsymbol{e}_n)-
\boldsymbol{u}_n\to\boldsymbol{0}$.
\item
\label{t:2013-08-19ii}
There exists a solution
$(\overline{\boldsymbol{x}},\overline{\boldsymbol{u}})$ to 
\eqref{e:999fgr623l17-27a} such that 
$\boldsymbol{x}_n\weakly\overline{\boldsymbol{x}}$ ~and~
$\boldsymbol{u}_n\weakly\overline{\boldsymbol{u}}$.
\end{enumerate}
\end{theorem}
\begin{proof}
Set 
\begin{equation}
\label{e:999fgr623l18-19k}
(\forall n\in\NN)\quad
\boldsymbol{z}_n=\boldsymbol{x}_n+\boldsymbol{u}_n
\quad\text{and}\quad
\boldsymbol{c}_n=P_{\boldsymbol{V}}\boldsymbol{e}_n-
P_{\boldsymbol{V}^\bot}\boldsymbol{e}_n. 
\end{equation}
Then $\sum_{n\in\NN}\lambda_n\|\boldsymbol{c}_n\|\leq
\sum_{n\in\NN}\lambda_n(\|P_{\boldsymbol{V}}\boldsymbol{e}_n\|+
\|P_{\boldsymbol{V}^\bot}\boldsymbol{e}_n\|)
\leq 2\sum_{n\in\NN}\lambda_n\|\boldsymbol{e}_n\|<\pinf$. 
Furthermore, since $(\boldsymbol{x}_n)_{n\in\NN}$ lies in 
$\boldsymbol{V}$ and $(\boldsymbol{u}_n)_{n\in\NN}$ lies in 
$\boldsymbol{V}^\bot$, \eqref{e:999fgr623l18-20b} can be 
rewritten as
\begin{equation}
\label{e:999fgr623l17-27b}
(\forall n\in\NN)\quad 
\begin{array}{l}
\left\lfloor
\begin{array}{l}
\boldsymbol{p}_n=J_{\boldsymbol{A}}(\boldsymbol{x}_n+
\boldsymbol{u}_n)+\boldsymbol{e}_n\\
\boldsymbol{r}_n=\boldsymbol{x}_n+\boldsymbol{u}_n-
\boldsymbol{p}_n\\
\boldsymbol{x}_{n+1}=
\boldsymbol{x}_n+\lambda_n(P_{\boldsymbol{V}}\boldsymbol{p}_n-
\boldsymbol{x}_n)\\
\boldsymbol{u}_{n+1}=
\boldsymbol{u}_n+\lambda_n(P_{\boldsymbol{V}^\bot}
\boldsymbol{r}_n-\boldsymbol{u}_n),
\end{array}
\right.\\[2mm]
\end{array}
\end{equation}
which yields
\begin{align}
\label{e:999fgr623l17-30b}
(\forall n\in\NN)\quad&
P_{\boldsymbol{V}}\bigg(\frac{\boldsymbol{z}_{n+1}-\boldsymbol{z}_n}
{\lambda_n}+\boldsymbol{z}_n-\boldsymbol{c}_n\bigg)
+P_{\boldsymbol{V}^\bot}\bigg(\boldsymbol{z}_n-\bigg(
\frac{\boldsymbol{z}_{n+1}-\boldsymbol{z}_n}
{\lambda_n}+\boldsymbol{z}_n-\boldsymbol{c}_n\bigg)\bigg)\nonumber\\
&=P_{\boldsymbol{V}}\bigg(\frac{\boldsymbol{z}_{n+1}
-\boldsymbol{z}_n}{\lambda_n}+\boldsymbol{z}_n-
\boldsymbol{c}_n\bigg)
+P_{\boldsymbol{V}^\bot}\bigg(\frac{\boldsymbol{z}_n-
\boldsymbol{z}_{n+1}}{\lambda_n}+\boldsymbol{c}_n\bigg)\nonumber\\
&=P_{\boldsymbol{V}}\bigg(\frac{\boldsymbol{x}_{n+1}-
\boldsymbol{x}_n}{\lambda_n}+\boldsymbol{x}_n
\bigg)+P_{\boldsymbol{V}^\bot}
\bigg(\frac{\boldsymbol{u}_n-\boldsymbol{u}_{n+1}}
{\lambda_n}\bigg)
-\boldsymbol{e}_n\nonumber\\
&=P_{\boldsymbol{V}}\boldsymbol{p}_n+
P_{\boldsymbol{V}^\bot}(\boldsymbol{u}_n-\boldsymbol{r}_n)
-\boldsymbol{e}_n\nonumber\\
&=P_{\boldsymbol{V}}\boldsymbol{p}_n+
P_{\boldsymbol{V}^\bot}(\boldsymbol{p}_n-\boldsymbol{x}_n)
-\boldsymbol{e}_n\nonumber\\
&=\boldsymbol{p}_n-\boldsymbol{e}_n\nonumber\\
&=J_{\boldsymbol{A}}\boldsymbol{z}_n.
\end{align}
Hence, it follows from \eqref{e:999fgr623l18-19k},
\eqref{e:999fgr623l17-27b}, and 
Lemma~\ref{l:999fgr623l17-29} that 
\begin{equation}
\label{e:999fgr623l18-20x}
(\forall n\in\NN)\quad\boldsymbol{z}_{n+1}=
\boldsymbol{z}_n+\lambda_n\big(J_{\boldsymbol{A}_{\boldsymbol{V}}}
\boldsymbol{z}_n+\boldsymbol{c}_n-\boldsymbol{z}_n\big). 
\end{equation}
Altogether, we derive from Lemmas~\ref{l:2013-07-26a} 
and~\ref{l:2009} that 
\begin{equation}
\label{e:999fgr623l18-22d}
J_{\boldsymbol{A}_{\boldsymbol{V}}}\boldsymbol{z}_n-\boldsymbol{z}_n
\to\boldsymbol{0}
\end{equation}
and that there exists 
$\boldsymbol{z}\in\zer\boldsymbol{A}_{\boldsymbol{V}}$ such that
\begin{equation}
\label{e:999fgr623l18-22c}
\boldsymbol{z}_n\weakly\boldsymbol{z}.
\end{equation}

\ref{t:2013-08-19i}: 
In view of \eqref{e:999fgr623l18-20b}, 
\eqref{e:999fgr623l18-19k}, Lemma~\ref{l:999fgr623l17-29}, 
and \eqref{e:999fgr623l18-22d}, we have
\begin{equation}
\label{e:999fgr623l18-21m}
\boldsymbol{x}_n-P_{\boldsymbol{V}}(\boldsymbol{p}_n-
\boldsymbol{e}_n)
=\boldsymbol{x}_n-P_{\boldsymbol{V}}J_{\boldsymbol{A}}
\boldsymbol{z}_n
=\boldsymbol{x}_n-P_{\boldsymbol{V}}
J_{\boldsymbol{A}_{\boldsymbol{V}}}\boldsymbol{z}_n
=P_{\boldsymbol{V}}\big(\boldsymbol{z}_n-
J_{\boldsymbol{A}_{\boldsymbol{V}}}\boldsymbol{z}_n\big)
\to\boldsymbol{0}
\end{equation}
and 
\begin{equation}
\label{e:999fgr623l18-21n}
\boldsymbol{u}_n-P_{\boldsymbol{V}^\bot}(
\boldsymbol{r}_n+\boldsymbol{e}_n)
=P_{\boldsymbol{V}^\bot}\boldsymbol{z}_n-P_{\boldsymbol{V}^\bot}
\big(\boldsymbol{z}_n-J_{\boldsymbol{A}}\boldsymbol{z}_n\big)
=P_{\boldsymbol{V}^\bot}J_{\boldsymbol{A}}\boldsymbol{z}_n
=P_{\boldsymbol{V}^\bot}\big(\boldsymbol{z}_n-
J_{\boldsymbol{A}_{\boldsymbol{V}}}\boldsymbol{z}_n\big)
\to\boldsymbol{0}.
\end{equation}

\ref{t:2013-08-19ii}: 
Since $\boldsymbol{z}\in\zer(\boldsymbol{A}_{\boldsymbol{V}})$,
it follows from Lemma~\ref{l:2013-07-26a} that
$(\overline{\boldsymbol{x}},\overline{\boldsymbol{u}})=
(P_{\boldsymbol{V}}\boldsymbol{z},P_{\boldsymbol{V}^\bot}
\boldsymbol{z})$ solves \eqref{e:999fgr623l17-27a}. Furthermore, 
using \eqref{e:999fgr623l18-19k}, \eqref{e:999fgr623l18-22c},
and the weak continuity of 
$P_{\boldsymbol{V}}$ and $P_{\boldsymbol{V}^\bot}$, we get
$\boldsymbol{x}_n=P_{\boldsymbol{V}}\boldsymbol{z}_n\weakly 
P_{\boldsymbol{V}}\boldsymbol{z}=\overline{\boldsymbol{x}}$ and 
$\boldsymbol{u}_n=
P_{\boldsymbol{V}^\bot}\boldsymbol{z}_n\weakly
P_{\boldsymbol{V}^\bot}\boldsymbol{z}=\overline{\boldsymbol{u}}$. 
\end{proof}

\begin{remark}
\label{r:2013-08-21}
Theorem~\ref{t:2013-08-19}\ref{t:2013-08-19ii} was obtained in 
\cite[Section~5]{Ecks92} under the additional 
assumptions that $\sum_{n\in\NN}\|e_n\|<\pinf$,
$\text{inf}_{n\in\NN}\lambda_n>0$, and 
$\text{sup}_{n\in\NN}\lambda_n<2$; the original weak convergence
result of \cite{Spin83} corresponds to the case when 
$(\forall n\in\NN)$ $\boldsymbol{e}_n=\boldsymbol{0}$ and 
$\lambda_n=1$.
As noted in \cite[Remark~2.2(iii)]{Joca09}, our assumptions 
do not require that $\sum_{n\in\NN}\|e_n\|<\pinf$ and they 
even allow for situations in which $(e_n)_{n\in\NN}$ does not 
converge weakly to $0$.
\end{remark}

\begin{remark}
\label{r:h5-hGEWD18}
It follows from Lemma~\ref{l:2013-07-26a} that any algorithm that
constructs a zero of $\boldsymbol{A}_{\boldsymbol{V}}$ can 
be used to solve \eqref{e:999fgr623l17-27a}. We have chosen 
to employ the proximal point algorithm of Lemma~\ref{l:2009}
because it features a flexible error model and it allows for 
under- and over-relaxations which can prove very useful in 
improving the convergence pattern of the algorithm. In addition,
\eqref{e:999fgr623l18-20x} leads to the simple implementation 
\eqref{e:999fgr623l18-20b} in terms of the resolvent 
$J_{\boldsymbol{A}}$. An alternative proximal point method is 
\begin{equation}
\label{e:alg14}
(\forall n\in\NN)\quad
\boldsymbol{z}_{n+1}=\boldsymbol{z}_n+\lambda_n
\big(J_{\gamma_n\boldsymbol{A}_{\boldsymbol{V}}}\boldsymbol{z}_n+
\boldsymbol{c}_n-\boldsymbol{z}_n\big),
\end{equation}
where $(\gamma_n)_{n\in\NN}$ lies $\RPP$ (weak convergence 
conditions under various hypotheses on $(\lambda_n)_{n\in\NN}$, 
$(\gamma_n)_{n\in\NN}$, and $(\boldsymbol{e}_n)_{n\in\NN}$ exist; 
see for instance \cite[Corollary~4.5]{Opti04} and 
\cite[Theorem~3]{Ecks92}).
However, due to the presence of the parameters
$(\gamma_n)_{n\in\NN}$, \eqref{e:alg14} results in an algorithm
which is much less straightforward to execute than 
\eqref{e:999fgr623l18-20b}. 
This issue is also discussed in \cite{Ecks92,Mahe95,Spin83}.
\end{remark}

\section{Primal-dual composite method of partial inverses}
\label{sec:3}

The following technical facts will be needed subsequently.

\begin{lemma}
\label{l:hj73aKi-24}
Let $\HH$ and $\GG$ be real Hilbert spaces, and let 
$L\in\BL(\HH,\GG)$. Set $\KKK=\HH\oplus\GG$ and
$\boldsymbol{V}=\menge{(x,y)\in\KKK}{Lx=y}$.
Then, for every $(x,y)\in\KKK$, the following hold:
\begin{enumerate}
\item
\label{l:hj73aKi-24i}
$P_{\boldsymbol{V}}(x,y)
=\big(x-L^*(\Id+LL^*)^{-1}(Lx-y),y+(\Id+LL^*)^{-1}(Lx-y)\big)$.
\item
\label{l:hj73aKi-24ii}
$P_{\boldsymbol{V}}(x,y)
=\big((\Id+L^*L)^{-1}(x+L^*y),L(\Id+L^*L)^{-1}(x+L^*y)\big)$.
\item
\label{l:hj73aKi-24iii}
$P_{\boldsymbol{V}^\bot}(x,y)
=\big(L^*(\Id+LL^*)^{-1}(Lx-y),-(\Id+LL^*)^{-1}(Lx-y)\big)$.
\item
\label{l:hj73aKi-24iv}
$P_{\boldsymbol{V}^\bot}(x,y)
=\big(x-(\Id+L^*L)^{-1}(x+L^*y),y-L(\Id+L^*L)^{-1}(x+L^*y)\big)$.
\end{enumerate}
\end{lemma}
\begin{proof}
Set
$\boldsymbol{T}\colon\KKK\to\GG\colon(x,y)\mapsto Lx-y$.
Then $\boldsymbol{T}\in\BL(\KKK,\GG)$, 
$\boldsymbol{T}^*\colon\GG\to\KKK\colon v\mapsto (L^*v,-v)$,
$\boldsymbol{T}\boldsymbol{T}^*=(\Id+LL^*)$ is invertible,
and $\boldsymbol{V}=\text{\rm ker}\,\boldsymbol{T}$ is a 
closed vector subspace of $\KKK$. Let us fix $(x,y)\in\KKK$.

\ref{l:hj73aKi-24i}:
Since $\boldsymbol{V}=\text{\rm ker}\,\boldsymbol{T}$, it follows 
from \cite[Example~28.14(iii)]{Livre1} that
\begin{align}
\label{e:genna2013-08-24s}
P_{\boldsymbol{V}}(x,y)
&=(x,y)-\boldsymbol{T}^*(\boldsymbol{T}\boldsymbol{T}^*)^{-1}
\boldsymbol{T}(x,y)\nonumber\\
&=\big(x-L^*(\Id+LL^*)^{-1}(Lx-y),y+(\Id+LL^*)^{-1}(Lx-y)\big).
\end{align}

\ref{l:hj73aKi-24i}$\Rightarrow$\ref{l:hj73aKi-24ii}: 
We have
\begin{align}
\label{e:genna2013-08-19b}
&\hskip -4mm
(\Id+L^*L)\big(x-L^*(\Id+LL^*)^{-1}(Lx-y)-(\Id+L^*L)^{-1}
(x+L^*y)\big)\nonumber\\
&=
x+L^*Lx-L^*(\Id+LL^*)^{-1}(Lx-y)-L^*LL^*(\Id+LL^*)^{-1}(Lx-y)-
x-L^*y\nonumber\\
&=L^*(Lx-y)-L^*(\Id+LL^*)^{-1}(Lx-y)-L^*LL^*(\Id+LL^*)^{-1}(Lx-y)
\nonumber\\
&=L^*\big(\Id+LL^*-\Id-LL^*\big)(\Id+LL^*)^{-1}(Lx-y)
\nonumber\\
&=0.
\end{align}
Therefore, 
$x-L^*(\Id+LL^*)^{-1}(Lx-y)=(\Id+L^*L)^{-1}(x+L^*y)$.
Likewise, 
\begin{align}
\label{e:genna2013-08-19o}
&\hskip -4mm
(\Id+LL^*)\big(y+(\Id+LL^*)^{-1}(Lx-y)-
L(\Id+L^*L)^{-1}(x+L^*y)\big)\nonumber\\
&=y+LL^*y+Lx-y-L(\Id+L^*L)^{-1}(x+L^*y)
-LL^*L(\Id+L^*L)^{-1}(x+L^*y)\nonumber\\
&=L(x+L^*y)-L(\Id+L^*L)^{-1}(x+L^*y)
-LL^*L(\Id+L^*L)^{-1}(x+L^*y)\nonumber\\
&=\big(L(\Id+L^*L)-L-LL^*L\big)(\Id+L^*L)^{-1}(x+L^*y)\nonumber\\
&=0
\end{align}
and hence $y+(\Id+LL^*)^{-1}(Lx-y)=L(\Id+L^*L)^{-1}(x+L^*y)$.

\ref{l:hj73aKi-24i}$\Rightarrow$\ref{l:hj73aKi-24iii}: 
$P_{\boldsymbol{V}^\bot}(x,y)=(x,y)-P_{\boldsymbol{V}}(x,y)
=\big(L^*(\Id+LL^*)^{-1}(Lx-y),-(\Id+LL^*)^{-1}(Lx-y)\big)$.

\ref{l:hj73aKi-24ii}$\Rightarrow$\ref{l:hj73aKi-24iv}: 
$P_{\boldsymbol{V}^\bot}(x,y)
=(x,y)-P_{\boldsymbol{V}}(x,y)
=\big(x-(\Id+L^*L)^{-1}(x+L^*y),y-L(\Id+L^*L)^{-1}(x+L^*y)\big)$.
\\ \hfill
\end{proof}

Our main algorithm is introduced and analyzed in the next theorem.
It consists of applying 
\eqref{e:999fgr623l18-20b} to the operator 
$\boldsymbol{A}\colon(x,y)\mapsto Ax\times By$ and the subspace 
$\boldsymbol{V}=\menge{(x,y)\in\HH\oplus\GG}{Lx=y}$.

\begin{theorem}
\label{t:2013-08-22}
In Problem~\ref{prob:1}, set $Q=(\Id+L^*L)^{-1}$ and
assume that $\zer(A+L^*BL)\neq\emp$.
Let $(\lambda_n)_{n\in\NN}$ be a sequence in $\left]0,2\right[$,
let $(a_n)_{n\in\NN}$ be a sequence in $\HH$, and let 
$(b_n)_{n\in\NN}$ be a sequence in $\GG$ such that 
$\sum_{n\in\NN}\lambda_n(2-\lambda_n)=\pinf$ and 
$\sum_{n\in\NN}\lambda_n\sqrt{\|a_n\|^2+\|b_n\|^2}<\pinf$. 
Let $x_0\in\HH$ and $v_0\in\GG$, and set
$y_0=Lx_0$, $u_0=-L^*v_0$, and 
\begin{equation}
\label{e:999fgr623l18-22x}
(\forall n\in\NN)\quad 
\begin{array}{l}
\left\lfloor
\begin{array}{l}
p_n=J_{A}(x_n+u_n)+a_n\\
q_n=J_{B}(y_n+v_n)+b_n\\
r_n=x_n+u_n-p_n\\
s_n=y_n+v_n-q_n\\
t_n=Q(r_n+L^*s_n)\\
w_n=Q(p_n+L^*q_n)\\
x_{n+1}=x_n-\lambda_n t_n\\
y_{n+1}=y_n-\lambda_n Lt_n\\
u_{n+1}=u_n+\lambda_n(w_n-p_n)\\
v_{n+1}=v_n+\lambda_n(Lw_n-q_n).
\end{array}
\right.\\[2mm]
\end{array}
\end{equation}
Then the following hold:
\begin{enumerate}
\item
\label{t:2013-08-22i}
$x_n-w_n+Q(a_n+L^*b_n)\to 0$ ~and~ 
$y_n-Lw_n+LQ(a_n+L^*b_n)\to 0$.
\item
\label{t:2013-08-22ii}
$u_n-r_n+t_n-a_n+Q(a_n+L^*b_n)\to 0$ ~and~ 
$v_n-s_n+Lt_n-b_n+LQ(a_n+L^*b_n)\to 0$.
\end{enumerate}
Moreover, there exists a solution $\overline{x}$ to 
\eqref{e:primal} and a solution $\overline{v}$ to 
\eqref{e:dual} such that the following hold:
\begin{enumerate}
\addtocounter{enumi}{2}
\item
\label{t:2013-08-22iii}
$-L^*\overline{v}\in A\overline{x}$ ~and~
$\overline{v}\in BL\overline{x}$.
\item
\label{t:2013-08-22iv}
$x_n\weakly\overline{x}$ ~and~ $v_n\weakly\overline{v}$.
\end{enumerate}
\end{theorem}
\begin{proof}
Set 
\begin{equation}
\label{e:genna2013-08-15A}
\KKK=\HH\oplus\GG\quad\text{and}\quad
\boldsymbol{V}=\menge{(x,y)\in\KKK}{Lx=y}
\end{equation}
and note that
\begin{equation}
\label{e:genna2013-08-15b}
\boldsymbol{V}^\bot=\menge{(u,v)\in\KKK}{u=-L^*v}.
\end{equation}
In addition, set 
\begin{equation}
\label{e:h5-hGEWD19a}
\boldsymbol{Z}=
\menge{(x,v)\in\KKK}{-L^*v\in Ax\;\;\text{and}\;\;v\in BLx}
\end{equation}
and
\begin{equation}
\label{e:genna2013-08-15H}
\boldsymbol{A}\colon\KKK\to 2^{\KKK}\colon(x,y)\mapsto Ax\times By.
\end{equation}
We also introduce the set
\begin{equation}
\label{e:999fgr623l18-23s}
\boldsymbol{S}=\menge{(\boldsymbol{x},\boldsymbol{u})\in
\boldsymbol{V}\times\boldsymbol{V}^\bot}{
\boldsymbol{u}\in\boldsymbol{A}\boldsymbol{x}}.
\end{equation}
Observe that
\begin{equation}
\label{e:h5-hGEWD19b}
\boldsymbol{S}=
\menge{\big((x,Lx),(-L^*v,v)\big)\in\KKK\times\KKK}
{(x,v)\in\boldsymbol{Z}}.
\end{equation}
Thus (see \cite{Siop11,Penn00} for the first two equivalences),
\begin{equation}
\label{e:h5-hGEWD21}
\zer\big(A+L^*BL\big)\neq\emp
\;\Leftrightarrow\;\zer\big(-LA^{-1}(-L^*)+B^{-1}\big)\neq\emp
\;\Leftrightarrow\;\boldsymbol{Z}\neq\emp
\;\Leftrightarrow\;\boldsymbol{S}\neq\emp.
\end{equation}
Now define $(\forall n\in\NN)$ 
$\boldsymbol{e}_n=(a_n,b_n)$, 
$\boldsymbol{p}_n=(p_n,q_n)$,
$\boldsymbol{r}_n=(r_n,s_n)$, 
$\boldsymbol{u}_n=(u_n,v_n)$, and 
$\boldsymbol{x}_n=(x_n,y_n)$. 
Then $\boldsymbol{x}_0\in\boldsymbol{V}$ and 
$\boldsymbol{u}_0\in\boldsymbol{V}^\bot$.
Moreover, by \cite[Proposition~23.16]{Livre1},
$\boldsymbol{A}$ is maximally monotone and 
\begin{equation}
\label{e:999fgr623l18-23t}
(\forall n\in\NN)\quad J_{\boldsymbol{A}}
(\boldsymbol{x}_n+\boldsymbol{u}_n)=
\big(J_A(x_n+u_n),J_B(y_n+v_n)\big).
\end{equation}
Furthermore, it follows from \eqref{e:genna2013-08-15A} and 
Lemma~\ref{l:hj73aKi-24}\ref{l:hj73aKi-24ii} that
\begin{equation}
\label{e:genna2013-08-15S}
(\forall n\in\NN)\quad P_{\boldsymbol{V}}\boldsymbol{r}_n
=\big(Q(r_n+L^*s_n),LQ(r_n+L^*s_n)\big),
\end{equation}
and from Lemma~\ref{l:hj73aKi-24}\ref{l:hj73aKi-24iv} that
\begin{equation}
\label{e:genna2013-08-15T}
(\forall n\in\NN)\quad P_{\boldsymbol{V}^\bot}\boldsymbol{p}_n
=\big(p_n-Q(p_n+L^*q_n),q_n-LQ(p_n+L^*q_n)\big).
\end{equation}
Thus, we derive from \eqref{e:999fgr623l18-23t}, 
\eqref{e:genna2013-08-15S}, and \eqref{e:genna2013-08-15T} that 
\eqref{e:999fgr623l18-20b} yields
\eqref{e:999fgr623l18-22x}. 
Altogether, since
$\sum_{n\in\NN}\lambda_n\|\boldsymbol{e}_n\|=
\sum_{n\in\NN}\lambda_n\sqrt{\|a_n\|^2+\|b_n\|^2}<\pinf$, 
Theorem~\ref{t:2013-08-19}\ref{t:2013-08-19i} and 
Lemma~\ref{l:hj73aKi-24} imply that \ref{t:2013-08-22i} and 
\ref{t:2013-08-22ii} are satisfied, and 
Theorem~\ref{t:2013-08-19}\ref{t:2013-08-19ii} implies that 
there exists $(\overline{\boldsymbol{x}},\overline{\boldsymbol{u}})
\in\boldsymbol{S}$ such that 
$\boldsymbol{x}_n\weakly\overline{\boldsymbol{x}}$ and
$\boldsymbol{u}_n\weakly\overline{\boldsymbol{u}}$.
Therefore, by \eqref{e:h5-hGEWD19b}, there exists
$(\overline{x},\overline{v})\in\boldsymbol{Z}$ such that
$(x_n,v_n)\weakly(\overline{x},\overline{v})$. 
Since $\boldsymbol{Z}\subset
(\zer(A+L^*BL))\times(\zer(-LA^{-1}(-L^*)+B^{-1}))$ 
\cite[Proposition~2.8(i)]{Siop11}, the proof is complete.
\end{proof}

\begin{remark}
\label{r:hj73aKi-26}
In the special case when $A=0$ and $L$ has closed range, an 
algorithm was proposed in \cite{Penn02} to solve the primal
problem \eqref{e:primal}, i.e., to find a point in $\zer (L^*BL)$. 
It employs the method of partial inverses in $\GG$ for finding
$y\in V$ and $v\in V^\bot$ such that $v\in By$, where
$V=\ran L$, and then solves $Lx=y$. Each iteration of the
resulting algorithm requires the computation of the generalized 
inverse of $L$, which is numerically demanding.
\end{remark}

The following application of Theorem~\ref{t:2013-08-22} concerns 
multi-operator inclusions.

\begin{problem}
\label{prob:2}
Let $m$ be a strictly positive integer, and let
$\HH$ and $(\GG_i)_{1\leq i\leq m}$ be real 
Hilbert spaces. Let $z\in\HH$, let $C\colon\HH\to 2^{\HH}$
be maximally monotone, and, for every $i\in\{1,\ldots,m\}$,
let $B_i\colon\GG_i\to 2^{\GG_i}$ 
be maximally monotone, let $o_i\in\GG_i$, 
and let $L_{i}\in\BL(\HH,\GG_i)$. 
Solve
\begin{equation}
\label{e:2012-09-24p}
\text{find}\;\;\overline{x}\in\HH
\;\;\text{such that}\;\;z\in 
C\overline{x}+\Sum_{i=1}^mL_{i}^*B_i(L_i\overline{x}-o_i)
\end{equation}
together with the dual problem
\begin{multline}
\label{e:2012-09-24d}
\text{find}\;\;\overline{v}_1\in\GG_1,\ldots,\overline{v}_m
\in\GG_m\;\;\text{such that}\\
(\forall i\in\{1,\ldots,m\})\;
-o_i\in-L_iC^{-1}
\bigg(z-\Sum_{j=1}^mL_j^*\overline{v}_j\bigg)
+B_i^{-1}\overline{v}_i.
\end{multline}
\end{problem}

\begin{corollary}
\label{c:h5-hGEWD21}
In Problem~\ref{prob:2}, set $Q=(\Id+\sum_{i=1}^mL_i^*L_i)^{-1}$ 
and assume that 
$z\in\ran(C+\sum_{i=1}^mL_{i}^*B_i(L_i\cdot-o_i)$.
Let $(\lambda_n)_{n\in\NN}$ be a sequence in $\left]0,2\right[$
such that $\sum_{n\in\NN}\lambda_n(2-\lambda_n)=\pinf$,
let $(a_n)_{n\in\NN}$ be a sequence in $\HH$, and let $x_0\in\HH$.
For every $i\in\{1,\ldots,m\}$, let 
$(b_{i,n})_{n\in\NN}$ be a sequence in $\GG_i$,
let $v_{i,0}\in\GG_i$, and set $y_{i,0}\!=\!L_ix_0$.
Suppose that $\sum_{n\in\NN}\lambda_n\sqrt{\|a_n\|^2\!+\!
\sum_{i=1}^m\|b_{i,n}\|^2}<\pinf$, and 
set $u_0=-\sum_{i=1}^mL_i^*v_{i,0}$ and 
\begin{equation}
\label{e:h5-hGEWD20}
(\forall n\in\NN)\quad 
\begin{array}{l}
\left\lfloor
\begin{array}{l}
p_n=J_{C}(x_n+u_n+z)+a_n\\
r_n=x_n+u_n-p_n\\
\text{For}\;i=1,\ldots,m\\
\left\lfloor
\begin{array}{l}
q_{i,n}=o_i+J_{B_i}(y_{i,n}+v_{i,n}-o_i)+b_{i,n}\\
s_{i,n}=y_{i,n}+v_{i,n}-q_{i,n}\\
\end{array}
\right.\\[1mm]
t_n=Q(r_n+\sum_{i=1}^mL_i^*s_{i,n})\\
w_n=Q(p_n+\sum_{i=1}^mL_i^*q_{i,n})\\
x_{n+1}=x_n-\lambda_n t_n\\
u_{n+1}=u_n+\lambda_n(w_n-p_n)\\
\text{For}\;i=1,\ldots,m\\
\left\lfloor
\begin{array}{l}
y_{i,n+1}=y_{i,n}-\lambda_n L_it_n\\
v_{i,n+1}=v_{i,n}+\lambda_n(L_iw_n-q_{i,n}).
\end{array}
\right.\\[1mm]
\end{array}
\right.\\[2mm]
\end{array}
\end{equation}
Then there exists a solution $\overline{x}$ to \eqref{e:2012-09-24p}
and a solution $(\overline{v}_i)_{1\leq i\leq m}$ to 
\eqref{e:2012-09-24d} such that the following hold:
\begin{enumerate}
\item
\label{c:h5-hGEWD21i}
$z-\sum_{i=1}^mL_i^*\overline{v}_i\in C\overline{x}$ ~and~
$(\forall i\in\{1,\ldots,m\})$~ 
$\overline{v}_i\in B_i(L_i\overline{x}-o_i)$.
\item
\label{c:h5-hGEWD21ii}
$x_n\weakly\overline{x}$ ~and~ 
$(\forall i\in\{1,\ldots,m\})$ $v_{i,n}\weakly\overline{v}_i$.
\end{enumerate}
\end{corollary}
\begin{proof}
Set $A\colon\HH\to 2^{\HH}\colon x\mapsto -z+Cx$,
$\GG=\bigoplus_{i=1}^m\GG_i$, 
$L\colon\HH\to\GG\colon x\mapsto (L_ix)_{1\leq i\leq m}$,
and $B\colon\GG\to 2^{\GG}\colon
(y_i)_{1\leq i\leq m}\mapsto\cart_{\!i=1}^{\!m}B_i(y_i-o_i)$.
Then $L^*\colon\GG\to\HH\colon (y_i)_{1\leq i\leq m}
\mapsto\sum_{i=1}^mL_i^*y_i$ and Problem~\ref{prob:2} is therefore 
an instantiation of Problem~\ref{prob:1}. Moreover, 
\cite[Propositions~23.15 and 23.16]{Livre1} yield
\begin{equation}
\label{e:J}
J_A\colon x\mapsto J_C(x+z)\quad\text{and}\quad
J_B\colon (y_i)_{1\leq i\leq m}\mapsto 
\big(o_i+J_{B_i}(y_i-o_i)\big)_{1\leq i\leq m}.
\end{equation}
Now set $(\forall n\in\NN)$ 
$b_n=(b_{i,n})_{1\leq i\leq m}$, $q_n=(q_{i,n})_{1\leq i\leq m}$,
$s_n=(s_{i,n})_{1\leq i\leq m}$, $v_n=(v_{i,n})_{1\leq i\leq m}$, 
and $y_n=(y_{i,n})_{1\leq i\leq m}$. In this setting,
\eqref{e:999fgr623l18-22x} coincides with
\eqref{e:h5-hGEWD20} and the claims therefore follow from 
Theorem~\ref{t:2013-08-22}\ref{t:2013-08-22iii}\&%
\ref{t:2013-08-22iv}.
\end{proof}

The next application addresses a primal-dual structured
minimization problem.

\begin{problem}
\label{prob:3}
Let $m$ be a strictly positive integer, and let
$\HH$ and $(\GG_i)_{1\leq i\leq m}$ be real 
Hilbert spaces. Let $z\in\HH$, let $f\in\Gamma_0(\HH)$,
and, for every $i\in\{1,\ldots,m\}$,
let $g_i\in\Gamma_0(\GG_i)$, 
$o_i\in\GG_i$, and $L_{i}\in\BL(\HH,\GG_i)$. 
Solve the primal problem
\begin{equation}
\label{e:h5-hGEWD23p}
\minimize{x\in\HH}{f(x)+\sum_{i=1}^mg_i(L_ix-o_i)-\scal{x}{z}}
\end{equation}
together with the dual problem
\begin{equation}
\label{e:h5-hGEWD23d}
\minimize{v_1\in\GG_1,\ldots,\,v_m\in\GG_m}{f^*\bigg(
z-\sum_{i=1}^mL_i^*v_i\bigg)+\sum_{i=1}^m
\big(g^*_i(v_i)+\scal{v_i}{o_i}\big)}.
\end{equation}
\end{problem}

\begin{corollary}
\label{c:nych5-hGEWD27}
In Problem~\ref{prob:3}, set $Q=(\Id+\sum_{i=1}^mL_i^*L_i)^{-1}$
and assume that 
\begin{equation}
\label{e:h5-hGEWD23a}
z\in\ran\bigg(\partial f+\sum_{i=1}^m
L_i^*(\partial g_i)(L_i\cdot-o_i)\bigg),
\end{equation}
Let $(\lambda_n)_{n\in\NN}$ be a sequence in $\left]0,2\right[$
such that $\sum_{n\in\NN}\lambda_n(2-\lambda_n)=\pinf$,
let $(a_n)_{n\in\NN}$ be a sequence in $\HH$, and let $x_0\in\HH$.
For every $i\in\{1,\ldots,m\}$, let 
$(b_{i,n})_{n\in\NN}$ be a sequence in $\GG_i$,
let $v_{i,0}\in\GG_i$, and set $y_{i,0}=L_ix_0$.
Suppose that $\sum_{n\in\NN}\lambda_n\sqrt{\|a_n\|^2+
\sum_{i=1}^m\|b_{i,n}\|^2}<\pinf$, and 
set $u_0=-\sum_{i=1}^mL_i^*v_{i,0}$ and 
\begin{equation}
\label{e:nych5-hGEWD26}
(\forall n\in\NN)\quad 
\begin{array}{l}
\left\lfloor
\begin{array}{l}
p_n=\prox_{f}(x_n+u_n+z)+a_n\\
r_n=x_n+u_n-p_n\\
\text{For}\;i=1,\ldots,m\\
\left\lfloor
\begin{array}{l}
q_{i,n}=o_i+\prox_{g_i}(y_{i,n}+v_{i,n}-o_i)+b_{i,n}\\
s_{i,n}=y_{i,n}+v_{i,n}-q_{i,n}\\
\end{array}
\right.\\[1mm]
t_n=Q(r_n+\sum_{i=1}^mL_i^*s_{i,n})\\
w_n=Q(p_n+\sum_{i=1}^mL_i^*q_{i,n})\\
x_{n+1}=x_n-\lambda_n t_n\\
u_{n+1}=u_n+\lambda_n(w_n-p_n)\\
\text{For}\;i=1,\ldots,m\\
\left\lfloor
\begin{array}{l}
y_{i,n+1}=y_{i,n}-\lambda_n L_it_n\\
v_{i,n+1}=v_{i,n}+\lambda_n(L_iw_n-q_{i,n}).
\end{array}
\right.\\[1mm]
\end{array}
\right.\\[2mm]
\end{array}
\end{equation}
Then there exists a solution $\overline{x}$ to \eqref{e:h5-hGEWD23p}
and a solution $(\overline{v}_i)_{1\leq i\leq m}$ to 
\eqref{e:h5-hGEWD23d} such that 
$z-\sum_{i=1}^mL_i^*\overline{v}_i\in\partial f(\overline{x})$,
$x_n\weakly\overline{x}$, and
$(\forall i\in\{1,\ldots,m\})$ 
$v_{i,n}\weakly\overline{v}_i\in\partial g_i(L_i\overline{x}-o_i)$.
\end{corollary}
\begin{proof}
Set $C=\partial f$ and $(\forall i\in\{1,\ldots,m\}$
$B_i=\partial g_i$. Then, using the same type of argument as 
in the proof of \cite[Theorem~4.2]{Svva12}, we derive from 
\eqref{e:h5-hGEWD23a} that Problem~\ref{prob:2} reduces to 
Problem~\ref{prob:3} and that \eqref{e:h5-hGEWD20} reduces to
\eqref{e:nych5-hGEWD26}. Thus, the assertions follow from 
Corollary~\ref{c:h5-hGEWD21}.
\end{proof}

\section{Alternative composite primal-dual method of partial 
inverses}
\label{sec:4}

The partial inverse method \eqref{e:999fgr623l18-22x} relies 
on the implicit assumption that the operator $(\Id+L^*L)^{-1}$ 
is relatively easy to implement. In some instances, it may
be advantageous to work with $(\Id+LL^*)^{-1}$ instead. 
In this section we describe an alternative method tailored to 
such situations.

\begin{theorem}
\label{t:nych5-hGEWD29}
In Problem~\ref{prob:1}, set $R=(\Id+LL^*)^{-1}$ and
assume that $\zer(A+L^*BL)\neq\emp$.
Let $(\lambda_n)_{n\in\NN}$ be a sequence in $\left]0,2\right[$,
let $(a_n)_{n\in\NN}$ be a sequence in $\HH$, and let 
$(b_n)_{n\in\NN}$ be a sequence in $\GG$ such that 
$\sum_{n\in\NN}\lambda_n(2-\lambda_n)=\pinf$ and 
$\sum_{n\in\NN}\lambda_n\sqrt{\|a_n\|^2+\|b_n\|^2}<\pinf$. 
Let $x_0\in\HH$ and $v_0\in\GG$, and set
$y_0=Lx_0$, $u_0=-L^*v_0$, and 
\begin{equation}
\label{e:nych5-hGEWD29x}
(\forall n\in\NN)\quad 
\begin{array}{l}
\left\lfloor
\begin{array}{l}
p_n=J_{A}(x_n+u_n)+a_n\\
q_n=J_{B}(y_n+v_n)+b_n\\
r_n=x_n+u_n-p_n\\
s_n=y_n+v_n-q_n\\
t_n=R(Lr_n-s_n)\\
w_n=R(Lp_n-q_n)\\
x_{n+1}=x_n+\lambda_n(L^*t_n-r_n)\\
y_{n+1}=y_n-\lambda_n(t_n+s_n)\\
u_{n+1}=u_n-\lambda_n L^*w_n\\
v_{n+1}=v_n+\lambda_nw_n.
\end{array}
\right.\\[2mm]
\end{array}
\end{equation}
Then the conclusions of Theorem~\ref{t:2013-08-22} are true.
\end{theorem}
\begin{proof}
The proof is analogous to that of Theorem~\ref{t:2013-08-22} except
that we replace \eqref{e:genna2013-08-15S} by
\begin{equation}
\label{e:nych5-hGEWD29S}
(\forall n\in\NN)\quad P_{\boldsymbol{V}}\boldsymbol{r}_n
=\big(r_n-L^*R(Lr_n-s_n),s_n+R(Lr_n-s_n)\big),
\end{equation}
and \eqref{e:genna2013-08-15T} by
\begin{equation}
\label{e:nych5-hGEWD29T}
(\forall n\in\NN)\quad P_{\boldsymbol{V}^\bot}\boldsymbol{p}_n
=\big(L^*R(Lp_n-q_n),-R(Lp_n-q_n)\big)
\end{equation}
by invoking Lemma~\ref{l:hj73aKi-24}\ref{l:hj73aKi-24i}%
\&\ref{l:hj73aKi-24iii}.
\end{proof}

Next, we present an application to a coupled inclusions problem.
This problem has essentially the same structure as
Problem~\ref{prob:2}, except that primal and dual inclusions 
are interchanged.

\begin{problem}
\label{prob:4}
Let $m$ be a strictly positive integer, and let
$(\HH_i)_{1\leq i\leq m}$ and $\GG$ be real 
Hilbert spaces. Let $o\in\GG$, let $D\colon\GG\to 2^{\GG}$
be maximally monotone, and, for every $i\in\{1,\ldots,m\}$,
let $z_i\in\HH_i$, let $A_i\colon\HH_i\to 2^{\HH_i}$ 
be maximally monotone, 
and let $L_{i}\in\BL(\HH_i,\GG)$. 
Solve the primal problem
\begin{equation}
\label{e:h5-hGEWD29p}
\text{find}\;\;\overline{x}_1\in\HH_1,\ldots,
\overline{x}_m\in\HH_m
\;\;\text{such that}\;\;
(\forall i\in\{1,\ldots,m\})\;
z_i\in A_i\overline{x}_i+L_i^*D\bigg(\Sum_{j=1}^mL_j
\overline{x}_j-o\bigg)
\end{equation}
together with the dual problem
\begin{equation}
\label{e:h5-hGEWD29d}
\text{find}\;\;\overline{v}\in\GG
\;\;\text{such that}\;\;
-o\in-\sum_{i=1}^mL_iA_i^{-1}
\big(z_i-L_i^*\overline{v}\big)
+D^{-1}\overline{v}.
\end{equation}
\end{problem}

\begin{corollary}
\label{c:nych5-hGEWD29}
In Problem~\ref{prob:4}, set $R=(\Id+\sum_{i=1}^mL_iL_i^*)^{-1}$ 
and assume that \eqref{e:h5-hGEWD29p} has at least one solution.
Let $(\lambda_n)_{n\in\NN}$ be a sequence in $\left]0,2\right[$
such that $\sum_{n\in\NN}\lambda_n(2-\lambda_n)=\pinf$,
let $(b_n)_{n\in\NN}$ be a sequence in $\GG$, and let $v_0\in\GG$.
For every $i\in\{1,\ldots,m\}$, let 
$(a_{i,n})_{n\in\NN}$ be a sequence in $\HH_i$,
let $x_{i,0}\in\HH_i$, and set 
$u_{i,0}=-L_i^*v_{0}$.
Suppose that $\sum_{n\in\NN}\lambda_n\sqrt{\|b_n\|^2+
\sum_{i=1}^m\|a_{i,n}\|^2}<\pinf$, and 
set $y_{0}=\sum_{i=1}^mL_ix_{i,0}$ and 
\begin{equation}
\label{e:nych5-hGEWD29a}
(\forall n\in\NN)\quad 
\begin{array}{l}
\left\lfloor
\begin{array}{l}
\text{For}\;i=1,\ldots,m\\
\left\lfloor
\begin{array}{l}
p_{i,n}=J_{A_i}(x_{i,n}+u_{i,n}+z_i)+a_{i,n}\\
r_{i,n}=x_{i,n}+u_{i,n}-p_{i,n}\\
\end{array}
\right.\\[1mm]
q_n=o+J_{D}(y_n+v_n-o)+b_n\\
s_n=y_n+v_n-q_n\\
t_n=R(\sum_{i=1}^mL_ir_{i,n}-s_n)\\
w_n=R(\sum_{i=1}^mL_ip_{i,n}-q_n)\\
\text{For}\;i=1,\ldots,m\\
\left\lfloor
\begin{array}{l}
x_{i,n+1}=x_{i,n}+\lambda_n(L_i^*t_n-r_{i,n})\\
u_{i,n+1}=u_{i,n}-\lambda_n L_i^*w_n\\
\end{array}
\right.\\[1mm]
y_{n+1}=y_n-\lambda_n(t_n+s_n)\\
v_{n+1}=v_n+\lambda_nw_n.
\end{array}
\right.\\[2mm]
\end{array}
\end{equation}
Then there exists a solution $(\overline{x}_i)_{1\leq i\leq m}$ to 
\eqref{e:h5-hGEWD29p} and a solution $\overline{v}$ to 
\eqref{e:h5-hGEWD29d} such that the following hold:
\begin{enumerate}
\item
\label{c:nych5-hGEWD29i}
$\overline{v}\in D(\sum_{i=1}^mL_i\overline{x}_i-o)$ ~and~
$(\forall i\in\{1,\ldots,m\})$~ 
$z_i-L_i^*\overline{v}\in A_i\overline{x}_i$.
\item
\label{c:nych5-hGEWD29ii}
$v_n\weakly\overline{v}$ ~and~ 
$(\forall i\in\{1,\ldots,m\})$ $x_{i,n}\weakly\overline{x}_i$.
\end{enumerate}
\end{corollary}
\begin{proof}
Set $B\colon\GG\to 2^{\GG}\colon y\mapsto D(y-o)$,
$\HH=\bigoplus_{i=1}^m\HH_i$, 
$L\colon\HH\to\GG\colon (x_i)_{1\leq i\leq m}\mapsto 
\sum_{i=1}^mL_ix_i$,
and $A\colon\HH\to 2^{\HH}\colon
(x_i)_{1\leq i\leq m}\mapsto\cart_{\!i=1}^{\!m}(-z_i+A_ix_i)$.
Then $L^*\colon\GG\to\HH\colon y\mapsto(L_i^*y)_{1\leq i\leq m}$ 
and hence Problem~\ref{prob:4} is a special case of 
Problem~\ref{prob:1}. On the other hand, 
\cite[Propositions~23.15 and 23.16]{Livre1} yield
\begin{equation}
\label{e:J'}
J_A\colon (x_i)_{1\leq i\leq m}\mapsto 
\big(J_{A_i}(x_i+z_i)\big)_{1\leq i\leq m}\quad\text{and}\quad
J_B\colon y\mapsto o+J_{D}(y-o).
\end{equation}
Now set $(\forall n\in\NN)$ 
$a_n=(a_{i,n})_{1\leq i\leq m}$, $p_n=(p_{i,n})_{1\leq i\leq m}$,
$r_n=(r_{i,n})_{1\leq i\leq m}$, $u_n=(u_{i,n})_{1\leq i\leq m}$, 
and $x_n=(x_{i,n})_{1\leq i\leq m}$. Then \eqref{e:nych5-hGEWD29x} 
reduces to \eqref{e:nych5-hGEWD29a} and we can appeal to 
Theorem~\ref{t:nych5-hGEWD29} to conclude.
\end{proof}

\begin{remark}
\label{r:nych5-hGEWD30}
In Problem~\ref{prob:4}, set 
$D=\partial g$, where $g\in\Gamma_0(\GG)$, and 
$(\forall i\in\{1,\ldots,m\}$ $A_i=\partial f_i$, 
where $f_i\in\Gamma_0(\HH_i)$ and 
$z_i\in\ran\big(\partial f_i+L_i^*(\partial g)(L_i\cdot-o)\big)$.
Then, arguing as in the proof of Corollary~\ref{c:nych5-hGEWD27}, 
we derive from Corollary~\ref{c:nych5-hGEWD29} an algorithm for 
solving the primal problem 
\begin{equation}
\label{e:nych5-hGEWD29P}
\minimize{x_1\in\HH_1,\ldots,\,x_m\in\HH_m}{
\sum_{i=1}^m\big(f_i(x_i)-\scal{x_i}{z_i}\big)
+g\bigg(\sum_{i=1}^mL_ix_i-o_i\bigg)}
\end{equation}
together with the dual problem
\begin{equation}
\label{e:nych5-hGEWD29D}
\minimize{v\in\GG}{\sum_{i=1}^mf_i^*\big(z_i-L_i^*v\big)
+g^*(v)+\scal{v}{o}}
\end{equation}
by replacing $J_{A_i}$ by $\prox_{f_i}$ and $J_D$ by $\prox_g$
in \eqref{e:nych5-hGEWD29a}.
\end{remark}

\section{Concluding remarks}
\label{sec:5}
We have shown that the method of partial inverses can be used 
to solve composite monotone inclusions in duality and 
have presented a few applications of this new framework. 
Despite their apparent complexity, all the algorithms developed 
in this paper are instances of the method of partial 
inverses, which is itself an instance of the proximal point 
algorithm. This underlines the fundamental nature of 
the proximal point algorithm and its far reaching ramifications.
Finally, let us note that in \cite{Luis12} the method of partial
inverses was coupled to standard splitting methods for the sum of
two monotone operators to solve inclusions of the form 
$0\!\in\!Ax\!+\!Bx\!+\!N_Vx$, where $N_V$ is the normal cone
operator of the closed vector subspace $V$.
Combining this approach to our results should lead to new splitting
methods for more general problems involving composite
operators, such as those studied in \cite{Svva12}.

\noindent
{\bfseries Acknowledgement.} 
This work was funded by the Deanship of Scientific
Research (DSR), King Abdulaziz University, under grant number 
4-130/1433 HiCi. 
The authors, therefore, acknowledge technical and
financial support of KAU.

\end{document}